\documentclass[12pt]{amsart}
\usepackage{amssymb,graphicx,amsmath}

\usepackage[stable]{footmisc}

\newtheorem{thm}{Theorem}
\newtheorem{prop}{Proposition}

\newtheorem{lem}{Lemma}

\newtheorem*{mainthm}{Main Theorem}

\theoremstyle{definition}
\newtheorem{defn}{Definition}

\def\om{\omega}
\def\Om{\Omega}

\usepackage[ marginparsep=5mm,heightrounded,centering,margin=3cm]{geometry}

\title[]{Uncountably many $2$-generated just-infinite \\ branch pro-2 groups}

\author{Mustafa G\"okhan  Benli, Rostislav Grigorchuk}
\date{3/19 /2015\\
The second author was supported by NSF grant DMS-1207699}

\address{Rostislav Grigorchuk\\  Texas A\& M University, College Station, Tx, Usa }
\email{grigorch@math.tamu.edu}

\address{Mustafa G\"okhan Benli\\  Middle East Technical University, Ankara, Turkey}
\email{benli@metu.edu.tr}
\begin{document}

   \begin{abstract}
   The aim of this note is to prove that there are $2^{\aleph_0}$ non-isomorphic 2 generated
    just-infinite branch pro-2 groups.
   \end{abstract}

  \maketitle
\section{Introdcution}

In 1937 B. Neumann \cite{Neumann_37} proved that there are $2^{\aleph_0}$  non-isomorphic $2$-generated abstract groups (see \cite[Chapter III]{MR1786869}). 
In 1983 the second author suggested a construction of continuously
many $2$-generated torsion $2$-groups $M_\om$ of intermediate growth and showed that among these
there are $2^{\aleph_0}$ distinct groups not only up to isomorphism but up to the weaker equivalence relation, quasi-isometry. One of the important features of the groups $M_\om$ is the fact that they are branch just-infinite groups.

Recall that a group is just-infinite if it is infinite but every proper quotient is finite.
Just-infinite groups lie on the border between finite and infinite groups. Every finitely generated
infinite group can be mapped onto a a just-infinite group. The class of branch groups
was introduced by the second author in \cite{MR1765119}. These groups act on a spherically
homogeneous rooted tree  such that the rigid stabilizers of levels have finite index 
(see next section for definitions). The fundamental fact about branch groups is that 
just-infinite branch groups constitute one of the three classes which the class of just infinite 
groups naturally splits \cite{MR1765119} (the other two classes are related to hereditarily just-infinite groups and simple groups).

Branch groups and just-infinite groups were also defined within the category of profinite groups
where now a dichotomy holds: Every just-infinite profinite group either is of branch type
or is related to hereditarily just-infinite groups \cite{MR1765120}.

A natural question is about the cardinality of the class of finitely generated profinite groups.
More precisely, one can fix the number of generators to $2$ and restrict the consideration to pro-$p$
groups where $p$ is a prime. Suprisingly, it seems that there are no results  in literature showing 
that the cardianlity of this class is $2^{\aleph_0}$. This question was raised recently in private conservations of the second author with A. Lubotzky and D. Segal. As a result, D. Segal suggested an elegant construction of $2^{\aleph_0}$ $2$-generator centre-by-metabelian pro-$p$ groups [D. Segal, private communication]. A. Lubotzky asked an analoguous question for finitely presented profinite groups which was answered in  \cite{snopce}, where for each $m\ge 3$, $2^{\aleph_0}$ non-isomorphic 
finitely presented metabelian  pro-$p$ groups are constructed. 

The main result of this paper is the following:

\begin{mainthm}
There exists $2^{\aleph_0}$ non-isomorphic $2$ generated just-infinite branch pro-$2$ groups.
\end{mainthm}

Unlike the groups of \cite{snopce}, our examples are not finitely presented (by arguments from \cite{MR2949211}), but have additional properties of being branch and just-infinite.  It is not known
if there are finitely presented branch groups. It would be interesting to know the cardinality
of the class of finitely generated hereditarily just-infinite pro-$p$ groups, that is, pro-$p$ 
groups with all subgroups of finite index being just-infinite groups.

The proof of the main theorem is based on a number of previous results, in particular the results of
L. Lavreniuk and V. Nekrashevych from  \cite{MR1890957}.

Although we do not do this here, the same techniques and construction of $p$-groups of \cite{MR784354}
may be used to prove them same result for the class of pro-$p$ groups, for any prime $p$.
\section{Preliminaries}

\subsection{Profinite groups and profinite completions} \mbox{}

\vspace{0.1cm}

We refer to \cite{MR2599132} for a detailed account on profinite groups. We recall some basic material related to profinite groups and completions.

Let $C$ b a class of finite groups. An inverse limit of a (surjective) inverse system of  groups in $C$
is called a pro-$C$ group. Equivalently, $G$ is a pro-$C$ group if it is compact, totally disconnected and  for any open normal subgroup $N$, $G/N$ lies in the class $C$. When the class $C$ is
the class of all finite groups (respectively, the class of all finite $p$-groups for a prime $p$), then
we will get the class of  profinite groups (respectively, pro-$p$ groups).

Given a group $G$, let $N_C(G)$ be the collection of finite index normal subgroups  $N$ of $G$, for which
$G/N$ lies in the class $C$. The pro-$C$ topology on $G$ is the topology having the collection 
$\{N \in N_C(G)\}$ as a neighbourhood basis around the identity element in $G$. This is the coarsest topology
on $G$ making it a topological group for which the canonical maps $G \to G/N, N \in N_C(G)$  are continuous (here  $G/N$ has the discrete topology). Assume that $N_C(G)$ is filtered from below, that is for any $N_1,N_2\in N_C(G)$ there is $N \in N_C(G)$ such that $N\le N_1 \cap N_2$ (observe that this
is true if $C$ is the class of finite groups or the class of finite $p$-groups, for a prime $p$). 
In this case, $N_C(G)$ is a directed set with order $N_1 \preceq N_2$ if $N_2 \subset N_1$, and
$\{G/N,P_{N,H},N\preceq H\in N_C(G)\}$ is an inverse system where $P_{N,H}:G/H \to G/N$ are the conical 
homomorphisms. The inverse limit of this inverse system is called the \textit{pro-C completion of} $G$ 
and is denoted by $\widehat{G}_{N_C(G)}$. If $C$ is the class of finite groups (resp. the class of finite $p$-groups) the profinite completion (resp. the pro-$p$ completion) of $G$ will be denoted by $\widehat{G}$ (resp. 
$\widehat{G}_p$).
If $G$ is residually $C$, that is $\bigcap_{n\in N_C(G)} N=1$, then one has a continuous embedding 
$G \to \widehat{G}_{N_C(G)}$ with dense image. 

\begin{defn}  \mbox{}

\begin{enumerate}
\item A group $G$ is called just-infinite if every proper quotient of $G$ is finite.
\item A profinite group $G$ is called just-infinite if every proper quotient by a closed normal
subgroup is finite.
\end{enumerate}
\end{defn}

Every finitely generated infinite group can be mapped onto a just-infinite group and every 
finitely generated pro-$p$ group can be mapped onto a just-infinite pro-$p$ group \cite{MR1765119}.
By an important result N. Nikolov and D. Segal \cite{MR2276769}, in a finitely generated profinite group a subgroup is open if and only if has finite index. Therefore, if $\phi:G \to H$ is a homomorphism
between profinite groups and $G$ is finitely generated, then $\phi$ is continuous.

\subsection{Groups of tree automorphisms and the congruence subgroup property} \mbox{}

\vspace{0.1cm}
We refer to \cite{MR1765119,MR2162164} for detailed accounts on groups acting on rooted trees.
Let $X=\{0,1,\ldots,d-1\}$ be a finite set with $d$ elements and let $X^\ast$ denote the set of finite words (sequences)  over $X$.
Elements of $X^\ast$ are naturally in bijection with the vertices of a rooted regular $d$-ary tree
in which the root is identified  with the empty word and vertices at distance $n$ from the root 
with words of length $n$. We will denote the length of an element $u\in X^\ast$ by $|u|$ and will 
say that $u$ lies in level $|u|$ of the tree.  We will not distinguish between $X^\ast$ and the tree it describes. 

The group $Aut(X^\ast)$ consists of all graph automorphisms of $X^\ast$. Equivalently, $Aut(X^\ast)$ consists
of all bijections of $X^\ast$ which fix the empty word and preserve prefixes, that is, if two words
have common prefix of length $n$, so do their images. It follows that an automorphism must preserve the
level of a vertex. We will say a subgroup $G\le Aut(X^\ast)$ is \textit{level transitive} if it acts
transitively on each level of the tree.

Given $g\in Aut(X^\ast)$ and $u\in X^\ast$,
the \textit{section} of $g$ at $u$ is the automorphism $g_u$ uniquely determined by
$$g(uv)=g(u)g_u(v)\;\; \text{for all} \; v \in X^\ast $$
There is an isomorphism  

\begin{equation} \label{wr}
\begin{array}{c}
Aut(X^\ast)\cong S_d \ltimes Aut(X^\ast)^d  \\
g\mapsto (\pi_g;g_0,\ldots,g_{d-1})
\end{array}
\end{equation}
 where $S_d$ is the symmetric group on $X$ and $\pi_g$ is the permutation determined 
by the action of $g$ onto $X$.

Let $X^{(n)}$ denote the set of words of length at most $n$ in $X^\ast$. $X^{(n)}$ describes the 
finite rooted tree of vertices up to level $n$. For each $n$, we have a map $r_n:Aut(X^{(n+1)}) \to Aut(X^{(n)})$ given by restriction.
It follows that $Aut(X^\ast)$ is isomorphic to the inverse limit of the inverse system 
$\{Aut(X^{(n)}),r_n\}$, and hence is a profinite group. Given distinct $g,h\in Aut(X^\ast)$,
define a metric (in fact an ultra metric) on $Aut(X^\ast)$ as follows: $d(g,h)=2^{-m(g,h)}$
where $m(g,h)$ is the maximum number $n$ such that the actions of $g$ and $h$ on $X^\ast$ agree up to
level $n$. A straightforward argument shows that  the topology generated by this metric is
the same as the topology on $Aut(X^\ast)$ as a profinite group. 
For $n\ge 1$, the $n$-th level stabilizer  $St(n)$ is the subgroup fixing (point wise) every vertex on 
level $n$. Each level has finitely many vertices, from which it follows that for each $n$, $St(n)$
is  of finite index and since $\bigcap_{ n\ge 1} St(n)=1$, $Aut(X^\ast)$ (and hence any subgroup of it)
is residually finite.

\begin{defn}
A subgroup $G\le Aut(X^\ast)$ is said to have the congruence subgroup property, if every finite index
subgroup of $G$ contains the subgroup $St_G(n)$ for some  $n$.
\end{defn}

If $G\le Aut(X^\ast)$ has the congruence subgroup
property then $\widehat{G}\cong \overline{G}$ where the latter denotes the closure of $G$ in $Aut(X^\ast)$ (see \cite[Theorem 9]{MR1765119} and  \cite[Lemma 1.1.9]{MR2599132}).

For $u\in X^\ast$  let us denote the subtree at $u$ by $uX^\ast$. The \textit{rigid stabilizer} of 
$G\le Aut(X^\ast)$  at $u$ is the subgroup $Rist_G(u)$ consisting of elements which act trivially outside of 
the sub-tree $uX^\ast$. The rigid stabilizer of level $n$ is the subgroup $Rist_G(n)=
\left< Rist_G(u) \mid |u|=n \right>= \prod_{|u|=n}Rist_G(u)$. 

\begin{defn}
A level transitive subgroup $G\le Aut(X^\ast)$ is called a branch group (resp. weakly branch group),
if for all $n$, the subgroup $Rist_G(n)$ has finite index in $G$
 (resp. is nontrivial).\end{defn}

The class  of branch groups plays an important role in the description of just-infinite groups and
contains many groups with unsual properties. We refer to \cite{MR1765119} for a detailed account on branch groups. We also note that many branch profinite groups have a universal embedding property
\cite{MR1754662}. The following gives criterion for a branch group to be just-infinite:

\begin{thm}\label{ji} \cite[Theorem 4]{MR1765119}
A branch group $G$ is just-infinite if and only if for each $u\in X^\ast$, $Rist_G(u)$ has finite
abelianization.
\end{thm}

\subsection{Grigorchuk groups \footnote{
 The first author insists on this terminology which is standard.}} \mbox{}

\vspace{0.1cm}

We recall a construction of groups from \cite{MR764305}.

Throughout this section we fix $X=\{0,1\}$. Let $\Om=\{0,1,2\}^\mathbb{N}$ be the set of all infinite sequences over $\{0,1,2\}$ and let $\sigma: \Om \to \Om$ be the shift given by
 $\sigma(\om_1\om_2\ldots)=\om_2\om_3\ldots$. For each $\om \in \Om$ we will define the automorphisms
 $b_\om,c_\om,d_\om$ recursively as follows:

For  $v\in X^\ast$
$$\begin{array}{llllll}
 b_\om(0v)=& 0 \beta(\om_1)(v) &   c_\om(0v)=& 0 \zeta(\om_1)(v)
  &  d_\om(0v)= &0 \delta(\om_1)(v) \\

   b_\om(1v)=& 1 b_{\tau (\om)}(v) &      c_\om(1v)= & 1 c_{\sigma \om}(v)
   & d_\om(1v)= & 1 d_{\sigma \om}(v), \\

\end{array}
$$
where
$$
\begin{array}{ccc}
 \beta(0)=a & \beta(1)=a & \beta(2)=e \\
  \zeta(0)=a & \zeta(1)=e & \zeta(2)=a \\
   \delta(0)=e & \delta(1)=a & \delta(2)=a \\
\end{array}
$$
and $e$ denotes the identity.  Also define the following automorphism $a$:
$$a(0v)=1v \;\text{and} \; a(1v)=0v  $$
Note that from the definition, the following relations are immediate:
$
 a^2=b_\om^2=c_\om^2=d_\om^2=b_\om c_\om d_\om =e
$.

For $\om=\om_1\om_2\ldots \in \Om$, let  $G_\om$ be the subgroup of $Aut(X^\ast)$ generated by $a,b_\om,c_\om,d_\om$.
The isomorphism (\ref{wr})  restricts to an embedding $G_\om \to S_2 \ltimes  (G_{\sigma \om} \times G_{
\sigma \om})$ for which

$$\begin{array}{cccr}
a & \mapsto & (01) &(e,e) \\
b_\om  & \mapsto &  &(\beta(\om_1),b_{\sigma \om}) \\
c_\om  & \mapsto &  &(\zeta(\om_1),c_{\sigma \om}) \\
d_\om  & \mapsto &  &(\delta(\om_1),d_{\sigma \om}) \\
\end{array}
$$

The groups $\{G_\om, \om \in \Om\}$ have a plethora of interesting and unusual properties related
to various notions such as growth and  amenability. We will mention some properties which will be used 
in the proof of the main theorem.

 Let $\Om_\infty \subset \Om$ be the subset consisting of sequences
in which the symbols $0,1,2$ appear infinitely often. 

\begin{thm}
For $\om \in \Om_\infty$, $G_\om$ is a just-infinite 2-group which is branch  and has the congruence subgroup
property.
\end{thm}

\begin{proof}
The facts that $G_\om$  is a $2$-group and is just infinite are established in 
\cite{MR764305}. The branch property is also implicitly proven and used in \cite{MR764305} (this property was not
defined at time of \cite{MR764305}). Explicit proofs can be found in \cite{perv_pre} or in 
\cite{BGN}. The congruence subgroup property is proven in \cite{MR1841763}.

\end{proof}

It was shown already in \cite{MR764305} that the family $\{G_\om \mid \om \in \Om\}$ contains
$2^{\aleph_0}$ non-isomorphic groups. The  complete  solution  of  isomorphism  problem  for  the 
groups  $G_\om$  was given by V. Nekrashevych \cite{MR2162164} based on the results of \cite{MR1890957}.

\begin{thm}\cite[Theorem 2.10.13]{MR2162164}
For $\om,\eta\in \Om_\infty$, $G_\om$ is isomorphic to $G_\eta$ if and only if $\omega$ 
is obtained from $\eta$ by an application of a permutation $\pi \in Sym\{0,1,2\}$ coordinate wise. 
\end{thm}

The ideas in the proof of this theorem will be  used  by  us  as  well.

\section{Proof of the main theorem}

Let $L_\om=\left<ab_\om, d_\om \right> \le G_\om$. It follows from the relations 
$
 a^2=b_\om^2=c_\om^2=d_\om^2=b_\om c_\om d_\om =e
$
that $L_\om$ is a normal subgroup of index $2$ in $G_\om$.

\begin{lem}
 Let $\om \in \Om, u \in X^\ast$ with $|u|=n$. Then, for all $g \in L_{\sigma^n \om}$
 there exists $h\in St_{L_\om}(u)$ such that $h_u=g$.

\end{lem}

\begin{proof}
Induction on $n$. 

Suppose $n=1$, then if $\om$ starts with $0$ we have 
$$\begin{array}{ccc}
(ab_\om)^2 & \mapsto  & (b_{\sigma \om}a, a b_{\sigma \om}) \\

(b_\om a)^2 & \mapsto  & (ab_{\sigma \om},  b_{\sigma \om}a) \\

d_\om & \mapsto  & (1 ,  d_{\sigma \om}) \\

ad_\om a & \mapsto  & (  d_{\sigma \om},1) \\
\end{array}$$
If $\om$ starts with 1 we have
$$\begin{array}{ccc}
(ab_\om)^2 & \mapsto  & (b_{\sigma \om}a, a b_{\sigma \om}) \\
(b_\om a)^2 & \mapsto  & (ab_{\sigma \om},  b_{\sigma \om}a) \\
d_\om & \mapsto  & (a ,  d_{\sigma \om}) \\
ad_\om a & \mapsto  & (  d_{\sigma \om},a) \\
\end{array}$$
and if $\om$ starts with 2 we have
$$\begin{array}{ccc}
(ab_\om)^2 & \mapsto  & (b_{\sigma \om},  b_{\sigma \om}) \\
d_\om & \mapsto  & (a ,  d_{\sigma \om}) \\
ad_\om a & \mapsto  & (  d_{\sigma \om},a) \\
d_\om (ab_\om)^2 & \mapsto & (a b_{\sigma \om} , c_{\sigma \om}) \\
ad_\om a (ab_\om)^2 & \mapsto & ( c_{\sigma \om} , a b_{\sigma \om} )
\end{array}$$
which show that the claim is true for $n=1$.
Assume now $n>1$ and let $u=vx$ for $x\in X$. By induction assumption, there is 
$k \in St_{L_{\sigma^{n-1}\om}}(x)$ such that $k_x=g$ and again by induction assumption
there is $h \in St_{L_\om}(v)$ such that $h_v=k$. Then clearly $h \in St_{L_\om}(u)$ and 
$h_u=h_{vx}=k_x=g$

\end{proof}

\begin{lem}
 The action of $L_\om$ is level transitive.
\end{lem}

\begin{proof}
Let $u,v \in X^\ast$ be of length $n$. By induction on $n$, if $n=1$, then $ab_{ \om}(u)=v$.
Suppose $n>1$ and let $xu,yv \in X^\ast$ be of length $n$. If $x=y$, 
 let $g\in L_{\sigma \om}$ such that $g(u)=v$ (such $g$ exists by the induction assumption). By the previous Lemma, there is  
$h \in St_{L_\om}(x)$ such that $h_x=g$. Then $h(xu)=h(x) h_x(u)=x g(u)=xv$. If $x\neq y$,
let $g\in L_{\sigma \om}$ such that $g\left(  (ab_\om)_x(u)\right)=v$. Again by the previous Lemma, 
there is
$h \in St_{L_\om}(y)$ such that $h_y=g$. Then $(h ab_\om)(xu)=h(y (ab_\om)_x(u))=
y g((ab_\om)_x(u))=yv$.
\end{proof}

\begin{lem}
 
 If  $\om \in \Om_\infty$ then $L_\om$ is a branch group.

\end{lem}

\begin{proof}
 
 Let $u \in X^\ast$ with $|u|=n$.
 Since $L_\om$ has finite index in
 $G_\om$, $Rist_{L_\om}(u)=Rist_{G_\om}(u) \cap L_\om$ has finite index in $Rist_{G_\om}(u)$. Therefore
 $Rist_{L_\om}(n)=\prod_{|u|=n} Rist_{L_\om}(u)$ has finite index in $Rist_{G_\om}(n)=\prod_{|u|=n} Rist_{G_\om}(u)$. Since $G_\om$ is branch, $Rist_{G_\om}(n)$ has finite index in $G_\om$ and hence
 $Rist_{L_\om}(n)$ has finite index in $L_\om$.
\end{proof}

\begin{lem}
  If  $\om \in \Om_\infty$,  $L_\om$ is just-infinite and has the congruence subgroup property.

\end{lem}

\begin{proof}
$G_\om$ has the congruence subgroup property and
hence $L_\om$ (having finite index in $G_\om$) has this property.
As mentioned in Theorem \ref{ji}, a branch group $G$ is just infinite if and only if each $Rist_G(u)$ has finite abelianization and clearly this is satisfied in $L_\om$ since
it is a periodic (torsion) group.
\end{proof}

It follows that the profinite completions $\widehat{L_\om}$ are isomorphic to the closures 
$\bar{L}_\om$ in $Aut(X^\star)$  and by  \cite[Theorem 2 and Corollary on Page 150]{MR1765119}, the groups $\widehat{L_\om}$
are just-infinite branch profinite (in fact pro-2) groups.
We will show that given $\om \in \Om_\infty,$ there is only one $\eta \in \Om_\infty\setminus \{\om\}$    such that $\widehat{L_\om}\cong \widehat{L_\eta}$, by using 
rigidity results and ideas of  \cite{MR1890957,MR2162164}.

\begin{defn}  \label{d1}
Let $G_1,G_2\le Aut(X^\ast)$ be level transitive. An isomorphism $\phi: G_1 \to G_2$ is called
saturated if there exists a sequence of subgroups $H_n \le G_1$ such that:

\begin{enumerate}
\item $H_n \le St_{G_1}(n)$ and $\phi(H_n)\le St_{G_2}(n)$,

\item for all $v \in X^\ast$ with $|v|=n$, $H_n$ and $\phi(H_n)$ act level transitively on the sub-tree
$vX^\ast$.
\end{enumerate}

\end{defn}

\begin{prop} \cite[Proposition 2.10.7]{MR2162164} Let $G_1,G_2\le Aut(X^\ast)$ be weakly branch groups and let $\phi: G_1 \to G_2$ be a saturated isomorphism. Then $\phi$ is induced by an automorphism of the
tree $X^\ast$.

\end{prop}

For $\om \in \Om$ define inductively $L_{0,\om}=L_\om$ and $L_{n,\om}=L_{n-1,\om}^2$ for $n\ge 1$.

\begin{lem} \label{five}
Let $\om \in \Om$ and $u\in X^\ast$ with $|u|=n$. Then for any $g \in L_{\sigma^n \om}$ there exists
$h \in St_{L_{n,\om}}(u)$ such that $h_u=g$.
\end{lem}

\begin{proof}
Induction on $n$. Suppose $n=1$ and $\om$ starts with $0$ (the other cases being similar). 
Let $r=(ab_\om)^2$ and $s=(ab_\om d_\om)^2=(ac_\om)^2$. Note  $r,s \in St_{L_{1,\om}}(1)$. We have
$$\begin{array}{ccc}
r & \mapsto  & (b_{\sigma \om}a,  ab_{\sigma \om}) \\
r^a & \mapsto  & (ab_{\sigma \om},  b_{\sigma \om}a) \\

(rs^{-1})^{ab_\om a} & \mapsto  &  (ad_{\sigma \om}, d_{\sigma \om}) \\

(rs^{-1})^{ab_\om} & \mapsto  &  (d_{\sigma \om}, ad_{\sigma \om}) \\
\end{array}$$
hence the claim is true for $n=1$.
Now suppose $n>1$ and let $u=vx$ for $v \in X^\ast$ and $x \in X$.
Let $g \in L_{\sigma^n \om}=L_{\sigma \sigma^{n-1}\om}$. By induction assumption, there exists 
$k \in St_{L_{1,\sigma^{n-1}\om}}(x)$ such that $k_x=g$. Since $k \in L_{1,\sigma^{n-1}\om}$, $k$
is of the form $$k=t_1^2\cdots t_m^2 \;\; \text{for some}\;\; t_i \in L_{\sigma^{n-1}\om}$$ 

Again by the induction assumption, there exits $h_i \in St_{L_{n-1,\om}}(v)$ such that $(h_i)_v=t_i$
for $i=1,2,\ldots,m$. Since $h_i \in L_{n-1,\om}$ we have $h_i^2 \in L_{n,\om}$. Now
$$(h_1^2\cdots h_m^2)_u=(h_1^2\cdots h_m^2)_{vx}=((h_1)_v^2\cdots (h_m)_v^2)_x=(t_1^2\cdots t_m^2)_x=k_x=g $$
\end{proof}

\begin{lem}\label{six} For all $\om \in \Om$ and $n$,

\begin{itemize}
\item[i)]  $L_{n,\om} \le St_{L_\omega}(n)$
\item[ii)]  For all $v\in X^\ast$ with $|v|=n$, $L_{n,\om}$ acts level transitively on the sub-tree $vX^\ast$. 
\end{itemize}

\end{lem}

\begin{proof}
The first assertion follows from a straightforward induction. For the second assertion, let $vu_1,vu_2
\in v X^\ast$ where $|u_1|=|u_2|$. Since $L_{\sigma^n\om}$ acts level transitively, there is 
$g \in L_{\sigma^n\om}$ such that $g(u_1)=u_2$. By Lemma \ref{five}  there exists $h\in St_{L_{n,\om}}(v)$ such that $h_v=g$. Hence $h(vu_1)=vu_2$.
\end{proof}

\begin{thm}
 Let $\om=\om_1\om_2\ldots$ and $\eta=\eta_1\eta_2\ldots$ be two sequences in $\Om_\infty$.  Then
 the groups $\widehat{L_\om}$ and $\widehat{L_\eta}$ are isomorphic if and only if 
  $\om=\eta$ or $\om_i= \pi(\eta_i)$ for all $i$,  where $\pi=(12) \in Sym\{0,1,2\}$. 
 
\end{thm}

\begin{proof}
If $\om_i= \pi(\eta_i)$ then by the definition of the groups,  $d_\om=d_\eta$ and 
$b_\om =c_\eta =b_\eta d_\eta $ and hence $ab_\om=ab_\eta d_\eta \in L_\eta$. Therefore, in 
this case $L_\om= L_\eta$ as subgroups of $Aut(X^\ast)$.

Now suppose that $\phi: \widehat{L_\om} \to \widehat{L_\eta}$ is an isomorphism. 
As mentioned above, we have $\widehat{L_\om}\cong \overline{L}_\om $ for any $\om \in \Om_\infty$ and 
 each $\overline{L}_\om$ is a branch group. Let $H_0=\overline{L}_\om$ and  $H_n=\overline {H_{n-1}^2}$ for $n\ge1$
and $K_0=\overline{L}_\eta $ and  $K_n=\overline {K_{n-1}^2}$ for $n\ge1$. By Lemma \ref{six} and continuity of $\phi$ we see that $\phi(H_n)=K_n$ and $H_n$  satisfies  the conditions of Definition \ref{d1} and hence $\phi$ is a saturated isomorphism.  It follows by the previous proposition that $\phi$ induced by an automorphism and hence $\bar{L}_\om$ and $\bar{L}_\eta$ are conjugate in $Aut(X^\ast)$. 

Given a vertex $v\in X^\ast$ and $g\in Aut(X^\ast)$, we say that $g$ is active at $v$ if $g_v$ 
acts non-trivially on the first the level.
 Consider the function $p: Aut(X^\ast)\to \mathbb{Z}_2^\mathbb{N}$  where  $p(g)_n$ is the number of active vertices for $g$ on level $n$ modulo $2$. It is straightforward to check that $p$ is a continuous group homomorphism
 (in fact one can check that this gives the abelianization). 
 It follows that $p(\overline{L}_\om)=p(\overline{L}_\eta)$. But by continuity we have $p(\overline{L}_\om)=
 \overline{ p(L_\om)}=p(L_\om)$ and hence $p(L_\om)=p(L_\eta)$. It remains to observe that one can 
 reconstruct the sequence $\om$ from $p(L_\om)$ up to the permutation $\pi$.
 
 We have $$p(ab_\om)=(1,\beta(\om_1),\beta(\om_2),\ldots) $$ $$p(d_\om)=(0,\delta(\om_1),\delta(\om_2),\ldots)$$
 where $\beta(0)=\beta(1)=1,\beta(2)=0$ and $\delta(0)=0,\delta(1)=\delta(2)=1.$
 Note that non-trivial elements of $p(L_\om)$ are $\{p(ab_\om),p(d_\om),p(ab_\om)+p(d_\om)\}$. 
 Given such a set of three sequences, the only sequence whose first entry is $0$ corresponds to $p(d_\om)$
 and one of the remaining ones corresponds to $p(ab_\om)$. If $\delta(\om_i)=0$ then $\om_i=0$.
 If $\delta(\om_i)=1$ then depending on the choice of the  sequence corresponding to $p(ab_\om)$, we will have either $w_i=1$ or $w_i=2$. 
\end{proof}

\bibliographystyle{alpha}
\bibliography{uprofinite}

\end{document}